\documentclass{amsart}

\usepackage{fullpage}

 \usepackage{amssymb,latexsym,amsmath,amsthm}

\usepackage[numbers,sort&compress]{natbib}

\usepackage{dsfont}

\newtheorem*{thm*}{Theorem A}

\newtheorem{thm}{Theorem}[section]

\newtheorem{dfn}{Definition}[section]

\newtheorem{lemma}{Lemma}[section]

\newtheorem{prop}{Proposition}[section]

\newtheorem*{conj*}{Conjecture A}

\newtheorem*{conj**}{Conjecture B}

 \numberwithin{equation}{section}

\begin{document}

\def\RR{{\mathbb{R}}}

\title[Symmetry for nonlinear nonlocal operators]{Symmetry properties for solutions of  nonlocal equations involving nonlinear operators}

\author{Mostafa Fazly and Yannick Sire}

\address{Department of Mathematics, The University of Texas at San Antonio, San Antonio, TX 78249, USA} \email{mostafa.fazly@utsa.edu}
\address{Department of Mathematics, Johns Hopkins University, Baltimore, MD 21218, USA} \email{sire@math.jhu.edu}

\thanks{}

\maketitle

\begin{abstract}  
We pursue the study of one-dimensional symmetry of solutions to nonlinear equations involving nonlocal operators. We consider a vast class of nonlinear operators and in a particular case it covers the fractional $p-$Laplacian operator.  Just like the classical De Giorgi's conjecture, we establish a Poincar\'e inequality and a linear Liouville theorem to provide two different proofs of the one-dimensional symmetry results in two dimensions. Both approaches are of independent interests.  In addition,  we provide certain  energy estimates for layer solutions and Liouville theorems for stable solutions. Most of the methods and ideas applied in the current article are applicable to nonlocal operators with general kernels where the famous extension problem, given by Caffarelli and Silvestre, is not necessarily known. 

\end{abstract}

\vspace{1cm}


\section{Introduction} 

We examine nonlocal and nonlinear operators whose
model is associated with  the following energy functional for $\Omega\subset \mathbb R^n$
\begin{equation}\label{energy}
\mathcal E^\Phi_K(u,\Omega):=\mathcal K^\Phi_K{(u,\Omega)} - \int_{\Omega}  F(u) dx, 
\end{equation}
when  the term $\mathcal K^\Phi_K$ is given by 
\begin{equation}\label{kphik}
\mathcal K^\Phi_K(u,\Omega):=  \frac{1}{2} \iint_{\mathbb R^{2n}\setminus (\mathbb R^n\setminus \Omega)^2 }  \Phi[ u(x) -u(y) ] K (x-y) dy dx. 
\end{equation}
For the above operator, we suppose that $K$ is a nonnegative  measurable kernel and even that is $K (z) = K (-z)$ for  $ z\in\mathbb R^n$ and $F\in C^1(\mathbb R)$.  We also assume that the function $\Phi\in C^2(\mathbb R^+)$, $\Phi(0)=\Phi'(0)=0$, $\Phi'$ is an odd function with $\Phi, \Phi', \Phi''>0$ in $\mathbb R^+$.   For even  nonlinearity $\Phi$ and even kernel $K$, the above kinetic energy $\mathcal K^\Phi_K$ becomes 
\begin{equation}\label{kphisymm}
\mathcal K^\Phi_K(u,\Omega) =  \frac{1}{2} \int_{\Omega} \int_{\Omega}   \Phi[ u(x) -u(y) ]  K (x-y) dy dx +  \int_{\Omega} \int_{\mathbb{R}^n\setminus \Omega}  \Phi[ u(x) -u(y) ]  K (x-y) dy dx. 
\end{equation}
The associated Euler-Lagrange nonlocal equation to (\ref{energy}) is 
 \begin{equation} \label{mainw}
\iint_{\mathbb R^{2n}}  \Phi'[ u(x) -u(y) ] [v(x) - v(y)] K (x-y) dy dx = \int_{\mathbb R^n} f(u(x)) v(x) dx,
  \end{equation}   
for every smooth function $v$ with compact support and when  $f(t)=F'(t)$. In this regard, we study solutions of the following nonlocal equation 
 \begin{equation} \label{main}
T_{ \Phi }[u(x)]  =  f(u(x))  \quad  \text{in} \ \  \RR^n , 
  \end{equation}   
  when the operator $T_{ \Phi }$ is defined by 
 \begin{eqnarray} \label{T}
 \begin{array}{lcl}
T_{ \Phi }[u(x)] &:= & p.v.  \int_{\RR^n } \Phi'[u(x) - u(y)] K (y-x) dy 
\\&=& 
\lim_{\epsilon\to 0} \int_{\RR^n\setminus B_\epsilon(x) } \Phi'[u(x) - u(y)] K (y-x) dy , 
\end{array}
    \end{eqnarray}   
where the notation $p.v.$ stands for the principal value.  Note that when  $\Phi(t)= \frac{t^2}{2}$ the operator  $T_{ \Phi }$  is a linear operator and the associated equation is of the form 
\begin{equation}\label{2lap}
p.v.\int_{\mathbb R^n}  [u(x)-u(y)]   K (y-x) dy  =   f(u)   \quad  \text{in} \ \  \mathbb{R}^n . 
\end{equation}
The above linear operator is well-studied in the literature in particular for the following (translation invariant) standard kernel 
 \begin{equation}\label{Jumpi}
 K (x-z) =  \frac{c(x-z)}{|x-z|^{n+\alpha}}, 
 \end{equation} 
where $c(x-z)$ is bounded between two positive constants $0<\lambda \le \Lambda$  and $0<\alpha<2$ (see \cite{fg} and references therein). For the case of $\lambda = \Lambda$  that is when 
\begin{equation}\label{fracK}
 K (x-z) =  \frac{\lambda}{|x-z|^{n+\alpha}}, 
 \end{equation} 
the operator in (\ref{2lap}) is known as the fractional Laplacian operator that is $(-\Delta)^{\alpha/2}$.  It is by now a well-known fact that the fractional Laplacian operator can be realized as the boundary operator (more precisely the Dirichlet-to-Neumann operator) of a suitable extension function in the half-space, see Caffarelli-Silvestre in \cite{cas}.   In addition to above kernels,  the following truncated kernels that are locally comparable to (\ref{Jumpi}) and have been of great interests as well 
 and with a  finite range 
 \begin{equation}\label{Jumpki}
\frac{c(x-z)}{|x-z|^{n+\alpha}} \mathds{1}_{\{|x-z|\le r_*\}} \le  K  (x-z) \le  \frac{c(x-z)}{|x-z|^{n+\alpha}} \mathds{1}_{\{|x-z|\le R_*\}}, 
 \end{equation} 
 when  $0< r_* \le R_*$.   We also consider the following kernel with decays that are 
\begin{equation}\label{Jumpkei}
K (x,z) = \frac{c (x-z)}{|x-z|^{n+\alpha }} \ \ \ \text{when} \ \ \ |x-z|\le R_* , 
 \end{equation} 
 and 
  \begin{equation}\label{Jumpkeri}
\int_{r<|x-z|<2r} |K   (x,z) | dz \le C D (r)\ \ \ \text{when} \ \ \ r > R_* ,  
 \end{equation} 
where  $0\le D \in C(\mathbb R^+)$ with $\lim_{r\to\infty} D (r)=0$.

The above problems  (\ref{energy}),  (\ref{main}) and (\ref{mainw})  have been of great interests in the literature for various nonlinearity $\Phi$ and kernel $K$.  For the case of general $\Phi$ satisfying 
\begin{equation*} 
\Phi(t) \le \Lambda |t| \ \ \ \text{and} \ \ \ \Phi(t) t \ge t^2,
\end{equation*}
and kernels modeled on the fractional $p$-Laplacian, several studies have been devoted to the regularity and a priori estimates for the weak solutions. In this regard we refer interested readers to  \cite{kms2,ckp1,ckp2} and references therein.  The fractional-type $p$-Laplacian is when $\Phi(t)= \frac{|t|^p}{p}$ for $p \ge 2$. The above equation (\ref{main}) is then of the following form
\begin{equation}\label{plap} 
\int_{\mathbb R^n}   |u(x)-u(y)|^{p-2} [u(x)-u(y)]   K (y-x) dy  =   f(u)   \quad  \text{in} \ \  \mathbb{R}^n . 
\end{equation}
The operator in the above equation with kernel (\ref{fracK}) for $\alpha=ps$ when $0<s<1$ is known as fractional $p$-Laplacian operator that is denoted usually $(-\Delta_p)^s$, see for instance \cite{bcf}. Various properties of solutions of this equation are studied extensively in the literature. Let us mention that there are many other  functions  fulfilling  conditions  on $\Phi$ such as  $\Phi(t)= \sqrt{1+t^2}-1$. For this choice of $\Phi$, the equation (\ref{main}) is of the following form
\begin{equation}\label{curv} 
 \int_{\mathbb R^n} \left[\frac{u(x) - u(y)}{\sqrt{1+|u(x)-u(y)|^2}  }  \right] K (y-x) dy  =   f(u)  \quad  \text{in} \ \  \mathbb{R}^n . 
  \end{equation} 
 The above equation can be seen as the fractional minimal graph equation.  Throughout the article, we assume that there exist positive constant $C$ and a nonnegative constant $\beta$  such that $\beta>\alpha$,  one of the following holds 
\begin{eqnarray}\label{Phiddtb}
\Phi''(t) &\le& C t^{\beta-2}  \ \ \ \text{for} \ \ \ t\in\mathbb R^+, 
\\ \label{Phidtb}
\Phi'(t) &\le& C t^{\beta-1}  \ \ \ \text{for} \ \ \ t\in\mathbb R^+, 
\\ \label{Phibeta}
\Phi(t) &\le& C  t^{\beta}  \ \ \ \text{for} \ \ \ t\in\mathbb R^+, 
\end{eqnarray}
when  $\beta \ge 2$.  Note that for equations (\ref{2lap}),  (\ref{plap}) and (\ref{curv}) the exponent  $\beta$ is given by $\beta=2$,  $\beta=p\ge 2$ and $\beta=2$,  respectively. 

In the current article, we also study the sum of nonlocal operators 
\begin{equation} \label{mainS}
S_{\Phi }[u(x)] = f(u(x))    \quad  \text{in} \ \  \RR^n , 
  \end{equation}   
  when the operator $S_{\Phi }$  stands for 
 \begin{equation*}\label{Ti}
S_{\Phi }[u(x)] : =  \sum_{i=1}^m T^i_{ \Phi_i }[u(x)] , 
  \end{equation*}   
where each $T^i_{ \Phi_i }$ is given by (\ref{}) that is 
 \begin{equation*}\label{TiK}
 T^i_{ \Phi_i }[u(x)] =  \sum_{i=1}^m  \lim_{\epsilon\to 0} \int_{\RR^n\setminus B_\epsilon(x) }  \Phi_i'[u(x) - u(y)] K_i (y-x) dy\ \ \ \ \text{for}\ \ \  m,n\ge1. 
  \end{equation*}  
Needless to say that for the case of $m=1$ operators given in (\ref{mainS}) and (\ref{main}) are equivalent.  The sum operators of the form (\ref{mainS}), and in particular the sum of fractional Laplacian operators, have been studied from both deterministic and probabilistic perspectives. In this regard we refer interested readers to \cite{cs} by Cabr\'{e} and Serra, to \cite{si} by  Silvestre, to \cite{bl} by Bass and Levin and references therein.  We assume that the truncated kernels $K_i$ are of the form 
  \begin{equation}\label{Jumpkis}
\frac{c_i(x-y)}{|x-y|^{n+\alpha_i}} \mathds{1}_{\{|x-z|\le r_i\}} \le  K_i  (x-y) \le  \frac{c_i(x-y)}{|x-y|^{n+\alpha_i}} \mathds{1}_{\{|x-y|\le R_i\}}, 
 \end{equation} 
for $0<r_i\le R_i$,  $\alpha_i>0$ and $0< \lambda_i  \le c_i \le \Lambda_i $.    Note that (\ref{Jumpkis}) is locally comparable to 
\begin{equation} \label{alKal2is}
K_i  (x-y) =   \frac{c_i(x-y)}{|x-y|^{n+\alpha_i}}  . 
\end{equation}
We assume that each $\Phi_i$ satisfies one of the conditions (\ref{Phiddtb})-(\ref{Phibeta}) for $\beta_i\ge 2 $ and $\beta_i>\alpha_i$.   As an example, consider operators $T^1_{\Phi_1}$,  $T^2_{\Phi_2}$ and  $T^3_{\Phi_3}$  given with (\ref{2lap}), (\ref{plap}) and (\ref{curv}), respectively. For these operators we have $\beta_1=2$,  $\beta_2=p$ and $\beta_3=2$ so that (\ref{Phibeta}) holds.  Now, consider  the following nonlocal problem 
  \begin{equation} \label{mainS123}
S_{\Phi }[u(x)]= \sum_{i=1}^3 T^i_{\Phi_i} [u(x)] = f(u(x))  )  \quad  \text{in} \ \  \RR^n. 
  \end{equation}   
Suppose also that $\alpha_1=2s$, $\alpha_2=ps$ and $\alpha_3=2s$ for $0<s<1$ and $p>2$ so that kernels $K_i$ satisfy all of the requirements on indices.   We also consider the following kernel with decays that is  
\begin{equation}\label{Jumpkeis}
K_i (x,z) = \frac{c_i (x-z)}{|x-z|^{n+\alpha_i }} \ \ \ \text{when} \ \ \ |x-z|\le R_i , 
 \end{equation} 
 and 
  \begin{equation}\label{Jumpkeris}
\int_{r<|x-z|<2r} |K_i   (x,z) | dz \le C D_i (r)\ \ \ \text{when} \ \ \ r > R_i ,  
 \end{equation} 
where each $0\le D_i \in C(\mathbb R^+)$ with $\lim_{r\to\infty} D_i (r)=0$.  

  The ideas and methods developed in the current article are strongly motivated by a famous conjecture of De Giorgi (1978) in \cite{DeGiorgi} that states bounded  monotone solutions of the  Allen-Cahn equation must be one-dimensional at least for $n \leq 8$. Here by monotonicity we mean monotonicity in one directions, e.g. $\partial_{x_n} u>0$.  The goal of the present article is to develop symmetry results for stable solutions of semilinear nonlocal equations involving  nonlinear operators described above in lower dimensions. The notion of stable solutions is as follows. 
  \begin{dfn} \label{stable}
A solution $u$ of (\ref{main}) is called stable when there exists  $\phi>0$   such that 
 \begin{equation} \label{L}
L_\Phi [\phi(x)]   =f'(u) \phi(x)   \ \ \ \text{in}\ \ \mathbb R^n, 
  \end{equation} 
  where $L_\Phi (\phi(x)) $ is the linearized operator and given by 
  \begin{equation*}\label{LPhi}
  L_\Phi [\phi(x)] := \lim_{\epsilon\to 0} \int_{\RR^n\setminus B_\epsilon(x) } \Phi''[u(x) - u(y)] [ \phi(x) - \phi(y)]   K (y-x) dy.
  \end{equation*}
\end{dfn} 
Note that stability is weaker assumption that monotonicity.  The De Giorgi's conjecture has been of great interests in the literature for the past decades from mathematical analysis, geometry and mathematical physics perspectives.  The conjecture was solved for $n=2$ by Ghoussoub-Gui in \cite{gg1}, for $n=3$ by Ambrosio-Cabr\'{e}  in \cite{ac} for the Allen-Cahn equation, and later by Alberti-Ambrosio-Cabr\'{e} \cite{aac} for a general nonlinearity.  In higher dimensions, up to some additional natural assumptions, the conjecture is settled by Savin in \cite{savin} and also by Ghoussoub-Gui in \cite{gg2}. In addition, we refer interested readers to  \cite{bar,bbg,fsv,bcn} for related results. A counterexample in dimensions $n\ge 9$ has been obtained by del Pino-Kowalczyk-Wei in \cite{dkw}.  Note also that De Giorgi type results for the case of fractional Laplacian operator are provided in \cite{cc,cc2,csol,hrsv,sv,cabSire} and references therein. 

Here is how this article is structured. In Section  \ref{SecPoi}, we establish a Poincar\'e type inequality for stable solutions of (\ref{main}) with a general kernel $K$. This inequality is inspired by the ones given originally by  Sternberg and Zumbrun in \cite{sz,sz1} and later in \cite{cf,fsv,fsv2,fg}.  In Section \ref{SecOne}, we apply the Poincar\'e inequality to establish our main result that is one-dimensional symmetry of solutions for  (\ref{main}) in two dimensions when the kernel is either with finite range or with decay at infinity.  In Section \ref{SecLio}, we prove a linear Liouville theorem and we apply the theorem to provide a second proof of our main results. In Section \ref{SecEne}, we prove certain energy estimates for layer solutions under various assumptions on kernels. Lastly, in Section \ref{SecSum}, we consider the sum of nonlocal operators examined in previous sections and we provide similar results.

\section{A Poincar\'{e} inequality for stable solutions}\label{SecPoi}

We start this section with a technical lemma that is useful in the forthcoming proofs. 

\begin{lemma}\label{lemtech}
Assume that  operators $T_\Phi$ and $L_{\Phi}$ are given by (\ref{T}) and (\ref{L}) with  a measurable and even  kernel  $K$.   Suppose also that $f,g\in C^1(\mathbb R^n)$. Then, 
\begin{equation}\label{}
\int_{\mathbb R^n} g(x)T_{\Phi}(f(x)) dx=\frac{1}{2} \int_{\mathbb R^n} \int_{\mathbb R^n}   \Phi'[f(x) - f(y)]  \left[g(x)-g(y)  \right] K(x-y) dx dy, 
\end{equation}
and 
\begin{equation}
\int_{\mathbb R^n} g(x)L_{\Phi}(f(x)) dx=\frac{1}{2} \int_{\mathbb R^n} \int_{\mathbb R^n} \Phi''[u(x) - u(y)]    [f(x) - f(y)]  \left[g(x)-g(y)  \right] K(x-y) dx dy. 
\end{equation}
\end{lemma}

\begin{proof} These are direct consequences of the fact that $\Phi'$ and $\Phi''$ are odd and even functions, respectively.  
\end{proof}

In what follows we establish a stability inequality that is our main tool to derive a priori estimates on stable solutions. Note that this inequality is valid for a general kernel $K$ and nonlinearity $\Phi$. 

\begin{prop}\label{propstab}  
  Let $u$ be a stable solution of (\ref{main}).  Then,  for any $\zeta\in C_c^1(\mathbb R^n)$, 
\begin{equation} \label{stability}
\int_{\RR^n}  f'(u) \zeta^2(x) dx \le  \frac{1}{2} \iint_{\RR^{2n}}  \Phi''[u(x) - u(y)]   [\zeta(x)- \zeta(y)]^2 K(x-y) dy dx . 
\end{equation} 
\end{prop}  
\begin{proof}
Let $u$ denote a stable solution of (\ref{main}). Then,  there exists a function $\phi$  such that   \begin{equation} \label{L12}
L_\Phi[\phi]= f'(u) \phi   \ \ \ \text{in}\ \ \mathbb R^n. 
  \end{equation} 
Multiply both sides with $\frac{\zeta^2}{\phi}$ where $\zeta$ is a  test functions. Therefore, 
\begin{equation*} \label{}
L_\Phi[\phi] \frac{\zeta^2}{\phi}=  f'(u) \zeta^2   \ \ \ \text{in}\ \ \mathbb R^n.
  \end{equation*} 
 From this and (\ref{L12}) we get 
\begin{equation} \label{LL1}
 \int_{\RR^n}  f'(u(x)) \zeta^2(x)  dx  \le  \int_{\RR^n} L_{\Phi} [\phi(x)] \frac{\zeta^2(x)}{\phi(x)} dx . 
  \end{equation} 
Applying Lemma \ref{lemtech} for the right-hand side of the above, we have
\begin{equation*}\label{Lphi}
\int_{\RR^n} L_\Phi[ \phi(x)] \frac{\zeta^2(x)}{\phi(x)} dx=
\frac{1}{2} \int_{\RR^n} \int_{\RR^n}  \Phi''[u(x) - u(y)]   [\phi(x) - \phi(y)] \left[ \frac{\zeta^2(x)}{\phi(x)}-  \frac{\zeta^2(y)}{\phi(y)} \right] K (x-y) dx dy. 
\end{equation*}
Note that for $a,b,c,d\in\mathbb R$ when $ab<0$ we have 
\begin{equation*}
(a+b)\left[  \frac{c^2}{a} +  \frac{d^2}{b}  \right] \le (c-d)^2    .
\end{equation*}
Since each $\phi$ is positive,  we have $\phi(x)\phi(z)>0$. Setting  $a=\phi(x)$, $b=-\phi(y)$, $c=\zeta(x)$ and  $d=\zeta(y)$ in the above inequality and from the fact that $ab=-\phi(x)\phi(y)<0$, we  conclude 
\begin{equation*}
[\phi(x) - \phi(y)] \left[ \frac{\zeta^2(x)}{\phi(x)}-  \frac{\zeta^2(y)}{\phi(y)} \right] \le [\zeta(x)- \zeta(y)]^2    .
\end{equation*} 
Note that $\Phi''$ is even and $\Phi''>0$ in $\mathbb R^+$.  Therefore, 
\begin{equation*}
\int_{\RR^n} L[\phi(x)] \frac{\zeta^2(x)}{\phi(x)} dx\le  \frac{1}{2} \int_{\RR^n} \int_{\RR^n}  \Phi''[u(x) - u(y)]  [\zeta(x)- \zeta(y)]^2 K(y-x) dy dx.\end{equation*} 
This together with (\ref{LL1}) complete the proof.

\end{proof}

We are now ready to establish a Poincar\'{e} type inequality for stable solutions. As mentioned earlier, the methods and ideas that we apply here are strongly motivated by the ones given in \cite{sz,sz1,cf,fsv,fsv2,fg} and references therein. Note that the following inequality is valid for a vast class of kernels $K$ and nonlinearities $\Phi$. Note also that the function $f$ does not appear in the inequality directly.  

\begin{thm}\label{thmpoin}
 Assume that  $n\ge 1$ and $ u$ is a stable solution of (\ref{main}).  Then,  
\begin{eqnarray}\label{poinsysm1}
&&\iint_{  \mathbb R^{2n}\cap  \{|\nabla_x u|\neq 0\}} \Phi''[u(x) - u(x+y)]     \mathcal A_y(\nabla_x u)  [\eta^2(x)+\eta^2(x+y)] K(y) dx dy 
\\&\le&  
\iint_{  \mathbb R^{2n}} \Phi''[u(x) - u(x+y)]  \mathcal B_y(\nabla_x u) [ \eta(x) - \eta(x+y) ] ^2 K(y) d x dy  , 
  \end{eqnarray} 
for any $\eta \in C_c^1(\mathbb R^{n})$ where 
\begin{eqnarray}\label{mathcalA}
\mathcal A_y(\nabla_x u) &:= & |\nabla_x u(x)|  |\nabla_x u(x+y)| -\nabla_x u(x)  \cdot \nabla_x u(x+y) , 
\\ \label{mathcalB}
 \mathcal B_y(\nabla_x u)  & :=& |\nabla_x u(x)| | \nabla_x u(x+y)| .
 \end{eqnarray}
 \end{thm}

\begin{proof} Suppose that  $u$ is a stable solution of (\ref{main}), Proposition \ref{propstab} implies that the stability inequality (\ref{stability}) holds.  Test the stability inequality on $\zeta(x)=|\nabla _x u(x)| \eta(x)$ where $\eta \in C_c^1(\mathbb R^n)$ is a test function to get 
 \begin{eqnarray*} \label{stability}
&& \int_{\RR^n} f'(u) |\nabla_x u (x)|^2   \eta^2(x)  dx 
\\&\le&  \frac{1}{2}  \iint_{\RR^{2n} }  \Phi''[u(x) - u(x+y)]  [ |\nabla_x u(x)|  \eta(x)-  |\nabla_x u(x+y)| \eta(x+y)]^2 K (y) dy dx. 
\end{eqnarray*} 
Expanding the right-hand side of the above inequality,  we get  
 \begin{eqnarray*} \label{Hijstability}
&& \int_{\RR^n} f'(u) |\nabla_x u (x)|^2   \eta^2(x)  dx 
 \\&\le&  
\frac{1}{2} \iint_{\RR^{2n}}  \Phi''[u(x) - u(x+y)]  |\nabla_x u(x)|^2  \eta^2(x) K(y) dy dx
\\&&+ \frac{1}{2} \iint_{\RR^{2n}}   \Phi''[u(x) - u(x+y)]  |\nabla_x u(x+y)|^2  \eta^2(x+y) K(y) dy dx
\\&& -  \iint_{\RR^{2n}}    \Phi''[u(x) - u(x+y)]   |\nabla_x u(x)| |\nabla_x u(x+y)| \eta(x) \eta(x+y) K(y) dy dx. 
\end{eqnarray*} 
We now apply the equation (\ref{main}). Note that for any index $1\le k\le n$ we have 
\begin{eqnarray*}
\partial_{x_k} T_\Phi[u(x)] &=& L_\Phi [ \partial_{x_k} u(x)]
\\&=&
\int_{\RR^n } \Phi''[u(x) - u(y)] [ \partial_{x_k} u(x) - \partial_{x_k} u(x+y)]   K(y-x) dy 
\\&=& f'(u) \partial_{x_k} u(x) . 
\end{eqnarray*}
Multiplying both sides of the above equation with $\partial_{x_k} u(x) \eta^2(x)$  and integrating we have
\begin{equation*}
 \int_{\mathbb R^{n}} f'(u) [\partial_{x_k} u(x)]^2  \eta^2(x)
dx =   \int_{\mathbb R^n} \partial_{x_k} u(x) \eta^2(x) L_\Phi [\partial_{x_k} u(x)   ] dx . 
\end{equation*}
From Lemma \ref{lemtech} we can simplify the right-hand side of the above as  
\begin{equation*}
 \frac{1}{2}  \iint_{\mathbb R^{2n}}  \Phi''[u(x) - u(x+y)] 
\left [ \partial_{x_k} u(x) \eta^2(x) -\partial_{x_k} u(x+y) \eta^2(x+y)\right ]
\left[ \partial_{x_k} u(x) -\partial_{x_k} u(x+y) \right] K(y) dx dy  .
\end{equation*}
Combining the above two equalities, we get 
\begin{eqnarray*}\label{HijIden}
&& \int_{\mathbb R^{n}} f'(u) |\nabla_x u(x)|^2  \eta^2(x)  dx 
\\&=&   \frac{1}{2}    \iint_{\mathbb R^{2n}}  \Phi''[u(x) - u(x+y)]  |\nabla_x u(x)|^2   \eta^2(x)  K(y) dx dy 
 \\&&+  \frac{1}{2}   \iint_{\mathbb R^{2n}}   \Phi''[u(x) - u(x+y)]  |\nabla_x u(x+y)|^2   \eta^2(x+y)  K(y) dx dy
 \\&&   -  \frac{1}{2}   \iint_{\mathbb R^{2n}}   \Phi''[u(x) - u(x+y)]  \nabla_x u(x) \cdot \nabla_x u(x+y)   \eta^2(x)  K(y) dx dy 
 \\&&   -  \frac{1}{2}   \iint_{\mathbb R^{2n}}   \Phi''[u(x) - u(x+y)]  \nabla_x u(x) \cdot \nabla_x u(x+y)   \eta^2(x+y)  K(y) dx dy   . 
\end{eqnarray*}
Combining this and (\ref{Hijstability}) we end up with  
\begin{eqnarray*}\label{HijIden2}
 &&  \iint_{\mathbb R^{2n}}   \Phi''[u(x) - u(x+y)]   |\nabla_x u(x)| |\nabla_x u(x+y)| \eta(x) \eta(x+y) K(y) dy dx 
\\&\le&  
  \frac{1}{2}   \iint_{\mathbb R^{2n}}   \Phi''[u(x) - u(x+y)]  \nabla_x u(x) \cdot \nabla_x u(x+y)  \left[ \eta^2(x) +\eta^2(x+y) \right]  K(y) dx dy  .
\end{eqnarray*}
Using the fact that $ \eta(x) \eta(x+y)  =\frac{1}{2} [\eta^2(x) +\eta^2(x+y)] -  \frac{1}{2} \left[ \eta(x) -\eta(x+y) \right]^2 $ and regrouping terms we get the desired result. 
 
\end{proof}

\section{One-dimensional symmetry: via a Poincar\'e inequality}\label{SecOne}

In this section, we apply the Poincar\'e inequality, given in former section, to establish one-dimensional symmetry results for bounded stable solution of  (\ref{main}) in two dimensions. Due to mathematical techniques and ideas that we apply in the proof, we assume that the kernel $K$ is of finite range or with certain decay at infinity. 
 
\begin{thm}\label{Thmain} Suppose that $u$  is a bounded stable solution of  (\ref{main}) in two dimensions  and (\ref{Phiddtb}) holds. Assume also that  the kernel $K$ satisfies  either (\ref{Jumpki}) or (\ref{Jumpkei}) and (\ref{Jumpkeri}) with $D(r)<C r^{-\theta}$ for $\theta>\beta+1$.    Then,  $u$ must be a one-dimensional function.   
  \end{thm}

\begin{proof} 
From the Poincar\'{e} inequality (\ref{poinsysm1}),  we have 
\begin{eqnarray}\label{PhiAB}
&&\iint_{  \mathbb R^{2n}\cap  \{|\nabla_x u|\neq 0\}} \Phi''[u(x) - u(x+y)]     \mathcal A_y(\nabla_x u)  [\eta^2(x)+\eta^2(x+y)] K(y) dx dy 
\\&\le& \nonumber C   \iint_{  \mathbb R^{2n}}  \Phi''[u(x) - u(y)]   \left[ \eta(x) - \eta(y) \right] ^2 K(x-y) d x dy,
  \end{eqnarray} 
where $\mathcal A_y(\nabla_x u):=  |\nabla_x u(x)|  |\nabla_x u(x+y)| -\nabla_x u(x)  \cdot \nabla_x u(x+y) \ge 0$ for all $x,y$ and C is a positive constant depending on $|| \nabla_x u||_{\infty}$. Since   $\Phi''$ satisfies (\ref{Phiddtb}), we have 
\begin{equation*}
\Phi''[u(x) - u(y)] \le C |  u(x) - u(y) |^{\beta-2}. 
\end{equation*}
for $\beta\ge 2$ and $\beta>\alpha$.  Note that $|u(x) - u(y)|\le C |x-y|$ when $C$ is a positive constant depending only on $||u||_{\infty}$. Therefore, 
\begin{equation*}
\Phi''[u(x) - u(y)] \le C |  x-y |^{\beta-2}. 
\end{equation*}   
From this,   (\ref{PhiAB}) and the assumptions on the kernel, we get 
\begin{eqnarray}\label{PhiAB1}
&&\iint_{  \mathbb R^{2n}\cap  \{|\nabla_x u|\neq 0\}} \Phi''[u(x) - u(x+y)]     \mathcal A_y(\nabla_x u)  [\eta^2(x)+\eta^2(x+y)] K(y) dx dy 
\\&\le& \label{PhiB2} C   \iint_{\mathbb R^{2n} }   \left[ \eta(x) - \eta(y) \right] ^2 |  x-y |^{\beta-2} K(x-y) d x dy.   
  \end{eqnarray} 
   We  now test the above inequality  on the following standard test function 
\begin{equation*}
\eta (x):=\left\{
                      \begin{array}{ll}
                        \frac{1}{2}, & \hbox{if $|x|\le\sqrt{R}$,} \\
                      \frac{ \log {R}-\log {|x|}}{{\log R}}, & \hbox{if $\sqrt{R}< |x|< R$,} \\
                       0, & \hbox{if $|x|\ge R$.}
                                                                       \end{array}
                    \right.
 \end{equation*} 
 Suppose that   $ \Omega_R:=\cup_{i=1}^6 \Omega^i_R
$ where 
\begin{eqnarray*}
&& \Omega^1_R:=B_{\sqrt R}\times (B_ R\setminus B_{\sqrt R}), \  \Omega^2_R:=(B_ R\setminus B_{\sqrt R})\times (B_ R\setminus B_{\sqrt R}),  \  \Omega^3_R:=(B_ R\setminus B_{\sqrt R})\times (\mathbb R^n\setminus B_R),\ 
\\&& \Omega^4_R:=B_{\sqrt R}\times (\mathbb R^n\setminus B_R) , \ \Omega^5_R:=B_{\sqrt R}\times B_{\sqrt R},  \
\Omega^6_R:=(\mathbb R^n\setminus B_R)\times (\mathbb R^n\setminus B_R)  .
\end{eqnarray*}
From the definition of test function $\eta$ we have $|\eta(x)-\eta(y)|=0$ on $\Omega^5_R$ and $\Omega^6_R$.  We now apply this in (\ref{PhiAB1}) to get 
\begin{eqnarray}\label{intIJR}
&&  \iint_{ \{\mathbb R^n \times B_{\sqrt R}\} \cap \{|\nabla_x u|\neq 0\} } \Phi''[u(x) - u(x+y)]    \mathcal A_y(\nabla_x u) K(y)  dx dy 
 \\&\le& \nonumber C \sum_{i=1}^4  \iint_{\Omega^i_R \cap |x-y|\le R_* }  \left[ \eta(x) - \eta(y) \right] ^2 |  x-y |^{\beta-2-n-\alpha} d x dy  
\\&& \nonumber +C \sum_{i=1}^4  \iint_{\Omega^i_R \cap |x-y| > R_* }  \left[ \eta(x) - \eta(y) \right] ^2 |  x-y |^{\beta-2} K(x-y) d x dy  \\&=:&\label{sumgam}
C \sum_{i=1}^4 I_i(R) + C \sum_{i=1}^4 J_i(R). 
  \end{eqnarray} 
Applying properties of the test function $\eta$ to compute an upper bound for each $ I_i(R)$ and $J_i(R)$.  
In this regard, we use the following straightforward inequality 
\begin{equation}\label{logab}
|\log b - \log a|^2 \le \frac{1}{ab} |b-a|^2,
\end{equation}
where $a,b\in\mathbb R^+$. We now consider various cases based on the domains.  

\noindent {\bf Case 1}: Let $(x,y)\in \Omega^1_R \cap |x-y|\le R_*$.   Without loss of generality, we assume that $x\in B_{\sqrt{R}}\setminus  B_{\sqrt{R}-R_*}$ and $y\in B_ {\sqrt{R}+R_*} \setminus B_{\sqrt R}$.  Note that $\eta(x)=\frac{1}{2}$ and $\eta(y)=1-\frac{\log |y|}{\log R}$. Applying (\ref{logab}) and the fact that $|x|< \sqrt R \le |y|$ we get 
\begin{eqnarray*}
|\eta(x) -\eta(y)|^2 &=& \frac{1}{\log^2 R} |\log |y| - \log \sqrt R|^2
\le \frac{1}{\log^2 R} \frac{1}{|y| \sqrt R}  | |y| -  \sqrt R|^2 
\\ &\le& \frac{1}{R\log^2 R}   | |y| -  |x||^2 \le \frac{1}{R\log^2 R}   | y -  x|^2   .
\end{eqnarray*}
From this for kernels satisfying (\ref{Jumpki}) we have
\begin{eqnarray*}
I_1(R) 
&\le & \frac{C }{R\log^2 R} \left[ \int_{B_{\sqrt R}\setminus B_{\sqrt R-R_*}} dx \right]\left[\int_{B_{R_*}}  |z|^{\beta-n-\alpha} dz \right]
\\&\le &  \frac{C}{\beta-\alpha} \frac{R_*^{\beta-\alpha}}{\sqrt R \log^2 R} , 
\end{eqnarray*} 
where we have used the assumptions $\beta-\alpha>0$ and $n=2$. We now consider kernels satisfying (\ref{Jumpkei})-(\ref{Jumpkeri}) with $D(r)< C r^{-\theta}$ when $\theta>\beta+1$. Therefore,  
 \begin{eqnarray}\label{J1R}
J_1(R) & \le & \frac{C}{R\log^2 R} \left[ \int_{ B_{\sqrt{R}} \setminus B_{\sqrt{R}-R_*} } dx\right] \left[ \sum_{k=1}^\infty \int_{k R_* <|z|<2k R_*} |z|^\beta K(z) dz \right]
\\&\le&\nonumber \frac{C R_*^{\beta-\theta}}{R\log^2 R}   \left[ \int_{ B_{\sqrt{R}} \setminus B_{\sqrt{R}-R_*} }  dx \right] \left[ \sum_{k=1}^\infty k^{\beta-\theta}\right] \le  \frac{C R_*^{\beta-\theta }}{\sqrt R \log^2 R}. 
\end{eqnarray}

\noindent {\bf Case 2}: Suppose that  $(x,y)\in \Omega^2_R \cap |x-y|\le R_*$.  Without loss of generality we assume that $|x|\le |y|$. Since $x,y\in B_ R\setminus B_{\sqrt R}$,  we have 
\begin{equation*}
|\eta(x) -\eta(y)|^2 = \frac{1}{\log^2 R} |\log |y| - \log |x||^2 \le \frac{1}{\log^2 R} \frac{1}{|x| |y|}  | |y| -  |x||^2 \le \frac{1}{|x|^2\log^2 R}   | y -  x|^2   .
\end{equation*}
From this for kernels (\ref{Jumpki})  we conclude  
\begin{eqnarray*}
I_2(R) &\le& 
\frac{C}{\log^2 R}  \int_{B_{ R}\setminus B_{\sqrt R}} \frac{1}{|x|^2} dx   \int_{B_{R_*}}  |z|^{\beta-n-\alpha} dz  
\le \frac{C }{\log^2 R}  \int_{\sqrt R}^R r^{n-3}dr \int_{0}^{R_*}  r^{\beta-1-\alpha} dr 
\\&\le &  \frac{C}{\beta-\alpha} \frac{R_*^{\beta-\alpha}}{\log R} , 
\end{eqnarray*} 
and again we have used the assumptions $\beta-\alpha>0$ and $n=2$. On the other hand, for kernels  satisfying (\ref{Jumpkei}) and (\ref{Jumpkeri}) with decay $D(r)< C r^{-\theta}$ when $\theta>3$  we have 
\begin{eqnarray}\label{J2R}
J_2(R) &\le&  \frac{C}{\log^2 R} \left[ \int_{B_{ R}\setminus B_{\sqrt R}} \frac{dx}{|x|^2} \right] \left[ \sum_{k=1}^\infty \int_{k R_*<|z|<2k R_*} |z|^\beta K(z) dz\right]
\\&\le& \nonumber \frac{C R_*^{\beta-\theta}}{\log^2 R} \left[  \int_{\sqrt R}^R r^{n-3}dr  \right]  \left[\sum_{k=1}^\infty k^{\beta-\theta}\right] \le  \frac{C R_*^{\beta-\theta }}{\log R}. 
\end{eqnarray}

\noindent {\bf Case 3}: Suppose that  $(x,y)\in \Omega^3_R \cap |x-y|\le R_*$. Without loss of generality we assume that $x\in B_{R}\setminus  B_{R-R_*}$ and $y\in B_ {{R}+R_*} \setminus B_{R}$ for large enough $R$.  Therefore,  $\eta(x)=1-\frac{\log |x|}{\log R}$ and $\eta(y)=0$. Applying (\ref{logab}) and the fact that $|x|<  R \le |y|$,  we have  
\begin{eqnarray*}
|\eta(x) -\eta(y)|^2 &=& \frac{1}{\log^2 R} |\log |x| - \log  R|^2 \le \frac{1}{\log^2 R} \frac{1}{|x| R}  | |x| -   R|^2  \le \frac{1}{|x|^2 \log^2 R}   | |y| -  |x||^2 
\\&\le& \frac{1}{|x|^2\log^2 R}   | y -  x|^2  . 
\end{eqnarray*}
We first assume that (\ref{Jumpki})  and we provide the following upper bound   
\begin{eqnarray*}
I_3(R) &\le&  
\frac{C}{\log^2 R}  \int_{B_{ R}\setminus B_{R-R_*}} \frac{1}{|x|^2} dx   \int_{B_{R_*}}|z|^{\beta-n-\alpha} dz 
\\&\le &  \frac{C}{\beta-\alpha} \frac{R_*^{\beta-\alpha}}{\log^2 R} . 
\end{eqnarray*} 
Assume that (\ref{Jumpkei}) and (\ref{Jumpkeri}) hold when $D(r)< C r^{-\theta}$ for $\theta>\beta+1$. Then,  
\begin{eqnarray}\label{J3R}
J_3(R) \le   \frac{C R_*^{\beta-\theta }}{\log R}.
\end{eqnarray}

\noindent {\bf Case 4}: Suppose that  $(x,y)\in \Omega^4_R$. Note that $\eta(x)=\frac{1}{2}$ and $\eta(y)=0$ and  $|x-y|>R-\sqrt R>R_*$ for   large enough $R$.  This implies that $I_4(R)=0$ for either (\ref{Jumpki}) or  (\ref{Jumpkei})-(\ref{Jumpkeri}).  Note also that $J_4(R)=0$ provided (\ref{Jumpki}).  Therefore, we assume  that the kernel satisfy the decay assumptions (\ref{Jumpkei})-(\ref{Jumpkeri}) and 
\begin{equation*}\label{J4R}
J_4(R) = \frac{1}{2} \int_{ B_{\sqrt R}}  dx   \sum_{k=1}^\infty \int_{k(R-\sqrt R)<|z|<2k (R-\sqrt R)}|z|^{\beta-2} K(z) dz \le \frac{C R}{(R-\sqrt R)^{\theta-\beta+2 }}  \sum_{k=1}^\infty k^{\beta-2-\theta } 
 \le  \frac{C}{ R^{\theta -\beta+1} }. 
\end{equation*} 
From the above  cases and (\ref{intIJR}), we get 
  \begin{equation*}\label{}
 \frac{1}{2} \iint_{ \{\mathbb R^2 \times B_{\sqrt R}\} \cap \{|\nabla_x u|\neq 0\} } \Phi''[u(x) - u(x+y)]      \mathcal A_y(\nabla_x u) K(y)  dx dy 
 \le\frac{C}{\beta-\alpha} \frac{R_*^{\beta-\alpha}}{\log R} \ \ \ \text{for large} \ R.
\end{equation*}
Sending $R\to\infty$ and applying the fact that $\mathcal A_y(\nabla_x u)\ge 0$ for all $x,y\in\mathbb R^2$, we get
 \begin{equation*}
  \Phi''[u(x) - u(x+y)]   \mathcal A_y(\nabla_x u) K(y) = 0 \ \ \text{a.e. for all} \ \  x,y\in\mathbb R^2. 
  \end{equation*}
   Since $u$ is not constant and $\Phi''$ is an even function, we have $\Phi''(u(x) - u(x+y) ) >0$.  Therefore,  $\mathcal A_y(\nabla_x u) = 0$ for all $x\in\mathbb R^2$ and $y\in B_{r_*}$. This implies that 
\begin{equation*}
|\nabla_x u(x)|  |\nabla_x u(x+y) |=\nabla_x u(x)  \cdot \nabla_x u(x+y) , 
\end{equation*}
when $|\nabla_x u|\neq 0$.   The above is equivalent to 
\begin{equation*}
u_{x_1}(x) u_{x_2}(x+y)= u_{x_1}(x+y) u_{x_2}(x), 
\end{equation*}
and 
  \begin{equation*}\label{unabla}
\nabla_x u(x) \cdot \nabla_x^\perp u(x+y)=0. 
\end{equation*}
This finishes the proof.    \end{proof}

For the rest of this section, we provide a Liouville theorem for solutions of (\ref{main}) under some sign assumptions on the function $f$. 

\begin{thm}\label{thmliouville}
Let $ u$ be a bounded solution of (\ref{main}) when the kernel $K$ satisfies    either (\ref{Jumpki}) or  (\ref{Jumpkei})-(\ref{Jumpkeri}) with $D(r)< C r^{-\theta}$ when $\theta>\beta+1$. Suppose that $\Phi$ satisfies (\ref{Phidtb}).   If $f(u) \ge0$ or $uf(u)\le 0$, then $u$ must be constant provided   $n\le \beta$.   
\end{thm}

\begin{proof}
Suppose that  $f(u) \ge0$.  Let $\eta$ be a test function and  multiply  (\ref{main}) with $(u(x) - ||u||_{\infty}) \eta^{2m}(x)$ and integrate to get 
\begin{equation*}
\int_{\mathbb R^n}  \Phi'[u(x) - u(y) ] (u(x) - ||u||_{\infty}) \eta^{2m}(x) T_\Phi [u(x)]  dx \le 0, 
\end{equation*}
for $m:=\frac{\beta}{2} \ge 1$.  We now apply the technical Lemma \ref{lemtech} to conclude 
\begin{equation}\label{mathcal1}
0 \ge \iint_{\mathbb R^{2n}} \Phi'[u(x) - u(y) ] [ (u(x)  - || u||_{\infty}) \eta^{2m}(x) - (u(y)  - || u||_{\infty}) \eta^{2m}(y)  ] K(x-y) dx dy  . 
\end{equation}
Adding and subtracting $u(y) \eta^{2m}(x)$ and  $u(x) \eta^{2m}(y)$ to above and applying the fact that $\Phi'$ is an odd function,  we get 
\begin{eqnarray*}\label{}
&& \iint_{\mathbb R^{2n}} \Phi'[u(x) - u(y) ] [ u(x)  - u(y) ] [  \eta^{2m}(x)+\eta^{2m}(y)  ] K(x-y) dx dy
 \\&\le& 4 ||u||_{\infty} \iint_{\mathbb R^{2n}} |\Phi'[u(x) - u(y)] | |\eta^{2m}(x)-\eta^{2m}(y)| K(x-y) dx dy . 
\end{eqnarray*}
Note that $ |\eta^{2m}(x)-\eta^{2m}(y)   | \le 2m  |\eta(x)-\eta(y)| |\eta^{2m-1}(x)+\eta^{2m-1}(y)|   $. This and above implies that 
\begin{eqnarray*}\label{}
&& \iint_{\mathbb R^{2n}} \Phi'[u(x) - u(y) ] [ u(x)  - u(y) ] [  \eta^{2m}(x)+\eta^{2m}(y)  ] K(x-y) dx dy
 \\&\le&C \iint_{\mathbb R^{2n}} |\Phi'[u(x) - u(y)] |  |\eta(x)-\eta(y)| [\eta^{2m-1}(x)+\eta^{2m-1}(y)] K(x-y) dx dy, 
\end{eqnarray*}
where $C$ is a positive constant and it is independent from $R$. Consider the standard test function $\eta$ when $\eta=1$ in $\overline {B_R}$ and $\eta=0$ in $\overline{\mathbb R^n\setminus B_{2R}}$ with $\eta\in C_c^1(\mathbb R^n)$ and $||\nabla \eta||_{\infty}<C R^{-1}$ in $\overline{B_{2R}\setminus B_R}$.  We now apply the H\"{o}lder inequality with exponent to get 
\begin{eqnarray}\label{uxuy}
&& \iint_{\mathbb R^{2n}} \Phi'[u(x) - u(y) ] [ u(x)  - u(y) ] [  \eta^{2m}(x)+\eta^{2m}(y)  ] K(x-y) dx dy
 \\&\le&C  \left[\iint_{\Gamma_R } |\Phi'[u(x) - u(y)] |^{ \frac{2m}{2m-1} }  [\eta^{2m}(x)+\eta^{2m}(y)] K(x-y) dx dy \right]^{\frac{2m-1}{2m}}
 \\&& \left[\iint_{\Gamma_R }  |\eta(x)-\eta(y)|^{2m}  K(x-y) dx dy \right]^{\frac{1}{2m}} ,
\end{eqnarray}
when 
\begin{equation}\label{gamma2}
\Gamma_R=\cup_{i=1}^6\Gamma^i_R \ \ \ \text{and each } \Gamma^i_R \ \text{is given by} 
\end{equation}
\begin{eqnarray*}\label{}
&&  \Gamma^1_R:=B_{ R}\times (B_ {2R}\setminus B_{R}),  \Gamma^2_R:=(B_ {2R}\setminus B_{ R})\times (B_ {2R}\setminus B_{ R}), \Gamma^3_R:=(B_ {2R}\setminus B_{ R})\times (\mathbb R^n\setminus B_{2R}),\ 
\\&& \label{gamma3}   \ \Gamma^4_R:=B_{ R}\times (\mathbb R^n\setminus B_{2R}), \ \Gamma^5_R:=B_{R}\times B_{R},  \ \Gamma^6_R:=(\mathbb R^n\setminus B_{2R})\times (\mathbb R^n\setminus B_{2R}). 
\end{eqnarray*}
Note that from the assumptions we have  $ |\Phi'[u(x) - u(y)] | \le |u(x) - u(y)]|^{2m-1} $. Multiplying both side of this with $|\Phi'[u(x) - u(y)] |^{2m-1}$, we conclude  
\begin{equation*}
|\Phi'[u(x) - u(y)] |^{  2m  } \le |u(x) - u(y)]|^{2m-1}  |\Phi'[u(x) - u(y)] |^{2m-1} .
\end{equation*}
From the oddness assumption on $\Phi'$, we have $t\Phi(t)\ge 0$ for $t\in\mathbb R$. Therefore, 
\begin{equation*}
|\Phi'[u(x) - u(y)] |^{  \frac{2m}{2m-1}  } \le  \Phi'[u(x) - u(y)]  [u(x) - u(y)]. 
\end{equation*}
Substituting this in (\ref{uxuy}) we conclude 
\begin{eqnarray}\label{uxuy2}
&& \iint_{\mathbb R^{2n}} \Phi'[u(x) - u(y) ] [ u(x)  - u(y) ] [  \eta^{2m}(x)+\eta^{2m}(y)  ] K(x-y) dx dy
 \\&\le&C  \label{uxuy3}  \left[\iint_{\Gamma_R}\Phi'[u(x) - u(y) ] [ u(x)  - u(y) ] [\eta^{2m}(x)+\eta^{2m}(y)] K(x-y) dx dy \right]^{\frac{2m-1}{2m}}
 \\&& \label{uxuy4}  \left[\iint_{\Gamma_R} |\eta(x)-\eta(y)|^{2m}  K(x-y) dx dy \right]^{\frac{1}{2m}}. \end{eqnarray}
We now provide an upper bound for (\ref{uxuy4}). Note that when $(x,y)\in\Gamma^5_R\cup\Gamma^6_R$,  we have $|\eta(x)-\eta(y)|=0$. 
\begin{eqnarray*}
\iint_{\Gamma_R} |\eta(x)-\eta(y)|^{2m}  K(x-y) dx dy &=& \sum_{i=1}^4 \iint_{\Gamma^i_R \cap \{|x-y|\le R_*\} } |\eta(x)-\eta(y)|^{2m}  K(x-y) dx dy 
\\&&+ \sum_{i=1}^4  \iint_{\Gamma^i_R \cap \{|x-y| > R_*\} } |\eta(x)-\eta(y)|^{2m}  K(x-y) dx dy 
\\&=& \sum_{i=1}^4 I_i(R) + \sum_{i=1}^4 J_i(R).  
\end{eqnarray*}
For $(x,y)\in \Gamma^4_R$,   we have $|\eta(x)-\eta(y)|=1$ and for $(x,y)\in\cup_{i=1}^3 \Gamma^i_R$,  we conclude 
\begin{equation*}
|\eta(x)- \eta(y)|^{\beta} \le C R^{-\beta} |x-y|^{\beta}.
\end{equation*}
We now consider each domain $\Gamma^i_R$ for $1\le i\le 3$ and $\Gamma^4_R$ separately to provide upper bounds for (\ref{uxuy4}).  

\noindent {\bf Case 1}: Suppose that  $(x,y)\in \Gamma^1_R $. Then, for kernels satisfying (\ref{Jumpki}) we obtain 
\begin{eqnarray*}
I_1(R)&\le& CR^{-\beta}  \int_{B_R\setminus B_{R-R_*}} \int_{B_{R+R_*}\setminus B_R}  |x-y|^\beta K(x-y) dy dx 
\\&\le& C R^{-\beta} \int_{B_R\setminus B_{R-R_*}} dx  \int_{B_{R_*}}|z|^{\beta-n-\alpha}  dz  \le   \frac{C {R_*}^{\beta-\alpha}}{\beta-\alpha}  R^{n-1-\beta}   . 
\end{eqnarray*}
Now suppose that (\ref{Jumpkei})-(\ref{Jumpkeri}) hold with $D(r)< C r^{-\theta}$ when $\theta>\beta+1$. Then,
 \begin{equation*}
J_1(R) \le  CR^{-\beta} \left[ \int_{B_R} dx \right]  \left[\sum_{k=1}^\infty \int_{k R_*<|z|<2k R_*} |z|^\beta K(z) dz \right]
\le \frac{C R_*^{\beta-\theta}}{R^\beta}  \left[ \int_{B_{R}} dx \right] \left[\sum_{k=1}^\infty k^{\beta-\theta}\right] \le  \frac{C R_*^{\beta-\theta } }{R^{\beta-n}} , 
\end{equation*} 
where we have used $\theta>\beta+1$ and $\beta-\alpha>0$.

\noindent {\bf Case 2}: Suppose that  $(x,y)\in \Gamma^2_R $. Then, whenever (\ref{Jumpki}) holds we have  
\begin{eqnarray}
I_2(R)&\le&  CR^{-\beta} \left[\int_{B_ {2R}\setminus B_{ R}} dx \right] \left[ \int_{B_{R}} |z|^{\beta-n-\alpha} dz \right]
\\ &\le&   \frac{C {R_*}^{\beta-\alpha}}{\beta-\alpha}  R^{n-\beta}   . 
\end{eqnarray}
For kernels satisfying (\ref{Jumpkei})-(\ref{Jumpkeri}),  the above estimate holds for $I_2(R)$  and 
\begin{equation}\label{J2R}
J_2(R) \le  CR^{-\beta}  \left[  \int_{B_ {2R}\setminus B_{ R}} dx \right] \left[  \sum_{k=1}^\infty \int_{k R_*<|z|<2k R_*} |z|^\beta K(z) dz  \right]\le  \frac{C R_*^{\beta-\theta } }{R^{\beta-n}} , 
\end{equation} 
where we have used $\theta>\beta+1$ and $\beta-\alpha>0$.  

\noindent {\bf Case 3}: Suppose that  $(x,y)\in \Gamma^3_R $. Just like the previous cases we first assume that  (\ref{Jumpki}) holds. Then, 
\begin{eqnarray*}
I_3(R)&\le&  C R^{-\beta} \int_{B_{2R}\setminus B_{2R-R_*}} \int_{B_{2R+R_*}\setminus B_{2R}}  |x-y|^\beta K(x-y) dy dx 
\\&\le& C R^{-\beta} \int_{B_{2R}\setminus B_{2R-R_*}} dx \int_{B_{R_*}}| z|^{\beta-n-\alpha}  dz 
\le   \frac{C R_*^{\beta-\alpha}}{\beta-\alpha}  R^{n-1-\beta}   .  
\end{eqnarray*}
When the kernel satisfies (\ref{Jumpkei})-(\ref{Jumpkeri}), then an upper bound of the  form (\ref{J2R})  holds for $J_3(R)$.

\noindent {\bf Case 4}: Suppose that  $(x,y)\in \Gamma^4_R $.  Note that $I_4(R)=J_4(R)=0$ whenever (\ref{Jumpki}) holds for large enough $R$.   We now assume that  (\ref{Jumpkei})-(\ref{Jumpkeri}) holds and we provide an estimate for $J_4(R)$. Note that $\eta(x)=1$ and $\eta(y)=0$ and  $|x-y|>R>R_*$, 
\begin{equation}\label{J4R}
J_4(R) = \left[ \int_{ B_{ R}}  dx \right] \left[  \sum_{k=1}^\infty \int_{k R<|z|<2kR} K(z) dz\right]  \le C R^{n-\theta}   \sum_{k=1}^\infty k^{-\theta} 
 \le  C  R^{n-\theta}  . 
\end{equation} 
From the assumption $n\le \beta$ and from the estimate (\ref{uxuy2}), we conclude 
 \begin{equation*}\label{}
\iint_{\Gamma_R} \Phi'[u(x) - u(y) ] [ u(x)  - u(y) ] [  \eta^{2m}(x)+\eta^{2m}(y)  ] K(x-y) dx dy\le C , 
\end{equation*}
where $C $ is a positive constant that is independent from $R$. From this and (\ref{uxuy2}),  we get 
\begin{equation*}\label{}
\iint_{\mathbb R^{2n}}   \Phi'[u(x) - u(y) ] [ u(x)  - u(y) ] K(x-y) dx dy =0.  
\end{equation*}
This implies that  $  \Phi'[u(x) - u(y) ] [ (u(x)  - (u(y) ] K(x-y) =0$ a.e. $(x,y)\in\mathbb R^n\times\mathbb R^{n}$.  From the assumptions, we have $\Phi'$ is an odd function and $\Phi'>0$ in $R^+$. This implies that  $ \Phi'[u(x) - u(y) ] [ (u(x)  - (u(y) ] \ge 0$ and equality occurs if and only if $u(x)=u(y)$ for $x\in\mathbb R^2$ and $y\in B_{r_*(x)}$. This implies that  $u$ is constant. Note that the case of $u f(u)\le 0$ is  similar and we omit the proof.
\end{proof}

\section{Liouville Theorem: Second proof of Theorem \ref{Thmain}}\label{SecLio}

We now provide a Liouville theorem for the quotient $\sigma:=\frac{\psi}{\phi}$ when $\psi:=\nabla u\cdot \nu$ for $\nu(x)=\nu(x',0):\RR^{n-1}\to \RR $ and $\phi$ solves the linearized system (\ref{L}). 
Note that for stable solutions  $u$ of (\ref{main}),   there exists a function $\phi$  such that 
\begin{equation}\label{phi}
L_\Phi[\phi(x)] = f'(u) \phi(x) . 
\end{equation}
  Differentiating (\ref{mainS}) with respect to $x$,  we get 
\begin{equation}\label{psisum}
L_\Phi[\psi(x)] = f'(u) \psi(x). 
\end{equation}
 From (\ref{psisum}) and the fact that $\psi =\sigma \phi$,  we have 
\begin{equation}\label{sigmaphisum}
L_\Phi[ \sigma (x) \phi (x)] = f'(u) \sigma(x) \phi(x). 
\end{equation}
Multiply (\ref{phi}) with $-\sigma $ and add with (\ref{sigmaphisum}) to get 
\begin{equation}\label{LLL}
L_\Phi[  \sigma(x) \phi(x) ]- \sigma(x) L_\Phi [\phi(x)]= 0 . 
\end{equation}
Note that for any two functions $g,h\in C^1(\mathbb R^n)$, the following technical identity holds  
\begin{eqnarray*}\label{idenLL}
L_{\Phi} [g(x)h(x)] &=& g(x) L_{\Phi}[h(x)] + h(x) L_{\Phi}[g(x)]\\&& \nonumber -   \int_{\mathbb R^n} \Phi''[u(x) - u(y)]    \left  [g(x) - g(y)  \right]  \left[h(x)-h(y)  \right] K(x-y) dy. 
\end{eqnarray*}
Combining (\ref{LLL}) and (\ref{idenLL}) for $h=\phi$ and $g=\sigma$, we conclude  
\begin{equation}\label{Lphisigma}
\phi (x) L_\Phi[\sigma(x)] - \int_{\RR^n} \Phi''[ u(x) - u(y)]  [\sigma(x)-\sigma(y)] [\phi(x)-\phi(y)] K(x-y) dy = 0. 
\end{equation}  
This implies that 
\begin{equation}\label{linPhi}
\int_{\RR^n} \Phi''\left[ u(x) - u(y) \right] \left( \sigma(x)- \sigma(y) \right) \phi(y)  K(x-y) dy  = 0 . 
\end{equation}

\begin{thm}\label{thmlione}
Suppose that $\sigma$ and $\phi$ satisfy (\ref{linPhi}) and $\phi$ does not change sign.   Assume also that 
\begin{equation}\label{Phiuxuy}
   \iint_{ \{\cup_{k=1}^4 \Gamma^k_R \}} \Phi''\left[ u(x) - u(y) \right]   [\sigma(x) +  \sigma(y)]^2  \phi(x) \phi(y) |x-y|^2 K(x-y)   dy dx \le C R^2 , 
 \end{equation}
where $\Gamma^k_R$ are given in (\ref{gamma2}). Then,  $\sigma$ must be constant. 
\end{thm}
 \begin{proof} Multiplying both sides of (\ref{linPhi}) with $\eta^2(x) \sigma(x) \phi(x)$ and integrating, we get  
 \begin{equation*}\label{}
\iint_{\RR^{2n}} \Phi''\left[ u(x) - u(y) \right] \left( \sigma(x)- \sigma(y) \right) \phi(x) \phi(y)  K(x-y) \eta^2 (x) dx dy  = 0 , 
\end{equation*}
for a test function $\eta\in C_c^1(\mathbb R^n).$  Rearranging terms and apply the fact that $\Phi''$ is an even function, we get 
    \begin{equation}\label{etasigmale0}
  \iint_{\RR^{2n}}  \Phi''\left[ u(x) - u(y) \right] [\eta^2(x) \sigma(x) -  \eta^2(y) \sigma(y)]    [\sigma(x)- \sigma(y)] \phi(x) \phi(y)  K(x-y)   dy dx = 0   . 
  \end{equation}
 Note that 
\begin{equation}\label{etaiden}
 [\eta^2(x) \sigma(x) -  \eta^2(y) \sigma(y)] = \frac{1}{2}  [\sigma(x) -  \sigma(y)][\eta^2(x) + \eta^2(y)] 
 + \frac{1}{2}  [\sigma(x) +  \sigma(y)][\eta^2(x) - \eta^2(y)] .
\end{equation}
Combining (\ref{etasigmale0}) and (\ref{etaiden}), we get 
\begin{eqnarray*}
\ \ \ \ 0\le I &:=&    \iint_{\RR^{2n}}   \Phi''\left[ u(x) - u(y) \right] [\sigma(x) -  \sigma(y)]^2 [\eta^2(x) + \eta^2(y)] \phi(x) \phi(y)  K(x-y)   dy dx 
\\& = &    \iint_{\RR^{2n}} \Phi''\left[ u(x) - u(y) \right]  [\sigma^2(x) -  \sigma^2(y)] [\eta^2(x) - \eta^2(y)]  \phi(x) \phi(y)  K(x-y)   dy dx  
\\&\le& C \left( \iint_{\RR^{2n}}   \Phi''\left[ u(x) - u(y) \right] [\sigma(x) -  \sigma(y)]^2 [\eta^2(x) + \eta^2(y)] \phi(x) \phi(y)  K(x-y)   dy dx  \right)^{1/2}
\\& & \left( \iint_{\RR^{2n}}    \Phi''\left[ u(x) - u(y) \right] [\sigma(x) +  \sigma(y)]^2 [\eta(x) - \eta(y)]^2 \phi(x) \phi(y)  K(x-y)   dy dx \right)^{1/2} 
\end{eqnarray*}
Note that in the above we have used the Cauchy-Schwarz inequality and $[\eta(x) + \eta(y)]^2 \le 2 [\eta^2(x) + \eta^2(y)]$ and 
\begin{equation*}\label{etaiden1}
 [\sigma^2(y) -  \sigma^2(x)] [\eta^2(x) - \eta^2(y)] = [\sigma(y) -  \sigma(x)] [\sigma(y) +  \sigma(x)][\eta(x) - \eta(y)] [\eta(x) + \eta(y)]   . 
 \end{equation*}
We now set to be the standard test function that is $\eta=1$ in $\overline {B_R}$ and $\eta=0$ in $\overline{\RR^n\setminus B_{2R}}$ with $||\nabla \eta||_{L^{\infty}(B_{2R}\setminus B_R)}\le C R^{-1}$.  Therefore, 
\begin{eqnarray*}
\ \ \ I^2&\le & C \left(  \iint_{  \{\cup_{k=1}^4\Gamma^k_R \} \ }   \Phi''\left[ u(x) - u(y) \right] [\sigma(x) -  \sigma(y)]^2 [\eta^2(x) + \eta^2(y)] \phi(x) \phi(y)  K(x-y)   dy dx  \right)
\\& &  \left(     \iint_{ \{\cup_{k=1}^4\Gamma^k_R\}  }  \Phi''\left[ u(x) - u(y) \right]  [\sigma(x) +  \sigma(y)]^2 [\eta(x) - \eta(y)]^2 \phi(x) \phi(y)  K(x-y)   dy dx \right)  
\\& =:& I(R) J(R) ,
\end{eqnarray*}
where domain decompositions $\Gamma^k_R$  are set in (\ref{gamma2}).   From the definition of $\eta$,  for $(x,y)$ in $\{\cup_{k=1}^4\Gamma^k_R \} $    we have 
  \begin{equation*}
 (\eta(x)-\eta(y))^2 \le C R^{-2} |x-y|^2  .    
   \end{equation*}
Note that $I(R)\le I$ and from the assumptions we have 
 \begin{equation*}
J(R) \le R^{-2}  \iint_{\cup_{k=1}^4 \Gamma^k_R} \Phi''\left[ u(x) - u(y) \right]  [\sigma(x) +  \sigma(y)]^2  \phi(x) \phi(y) |x-y|^2 K(x-y)   dy dx \le C. 
   \end{equation*}
This implies that $0\le I\le C$ and then $I(R)\le C$. Therefore, $I=0$.  This completes the proof. 
    \end{proof}

 Note that the above Liouville theorem can be applied to establish one-dimensional symmetry results for higher dimensions that is $n\ge 2$. One can simplify the assumption (\ref{Phiuxuy}) as what follows.  Since $|\nabla u|$ is globally bounded,  we conclude that $|\sigma|\le \frac{C}{\phi}$. This implies that 
\begin{equation*}
 [\sigma(x) +  \sigma(y)]^2 \le C \left(\frac{1}{\phi^2(x)} + \frac{1}{\phi^2(y)} \right). 
\end{equation*}
Therefore, 
\begin{equation*}
 [\sigma(x) +  \sigma(y)]^2  \phi(x) \phi(y)  \le C  \left( \frac{\phi(x)}{\phi(y)} + \frac{\phi(y)}{\phi(x)}  \right). 
\end{equation*}
Suppose now that  the following Harnack inequality holds for $\phi$ 
 \begin{equation*}\label{harn}
 \sup_{B_1(x_0)} \phi  \le C \inf_{B_1(x_0)} \phi, \ \ \text{for all} \ \ x_0\in\mathbb R^n. 
 \end{equation*}
  This implies that 
\begin{equation*}
 [\sigma(x) +  \sigma(y)]^2  \phi(x) \phi(y)  \le C. 
\end{equation*}
From this, the assumption (\ref{Phiuxuy}) can be simplified as  
\begin{equation}\label{Phiuxuysim}
   \iint_{ \{\cup_{k=1}^4 \Gamma^k_R \}} \Phi''\left[ u(x) - u(y) \right]  |x-y|^2 K(x-y)   dy dx \le C R^2.  
 \end{equation}
Let $u$  be a bounded monotone solution of   (\ref{main}) in two dimensions when the  kernel $K$ satisfies either (\ref{Jumpki}) or  (\ref{Jumpkei})-(\ref{Jumpkeri}) with $D(r)< C r^{-\theta}$ when $\theta>\beta+1$.  Applying similar arguments as in the proof of Theorem \ref{thmliouville}, one can conclude that (\ref{Phiuxuysim}) holds in two dimensions.  Therefore,   $u$ must be a one-dimensional function.  

We end this section with mentioning that  bounded global minimizers of nonlocal energy is studied in \cite{bucur}. The author has provided one-dimensional symmetry results for global energy minimizers of certain nonlocal operators in two dimensions,  under various assumptions on the operator. The ideas and methods applied in this article are different from ours, however, there are some connections in the spirit.

\section{Energy estimates for layer solutions}\label{SecEne}

Let us start this section with the notion of layer solutions. 
\begin{dfn}\label{layer}
We say that $u$ is a layer solution of (\ref{main}) if $ u$ is a bounded monotone solution of (\ref{main}) such that 
 \begin{equation}\label{asympsys}
\lim_{x_n\to \pm\infty}  u(x',x_{n})= \pm 1
 \ \ \text{for} \ \ \ x'\in\mathbb R^{n-1}. 
\end{equation}  
\end{dfn}

We refer interested readers to \cite{csol,cc,cabSire,cp,sv,savin} and references therein in regards to layer solutions. Note that assumption (\ref{asympsys}) is known as a natural assumption in this context and Savin's proof of De Giorgi's conjecture in dimensions $4\le n\le 8$ and the counterexample of del Pino-Kowalczyk-Wei in dimensions $n\ge 9$ rely on (\ref{asympsys}). The following theorem deals with energy estimates for  layer solutions of (\ref{main}) when the kernel is either with finite range or decay at infinity. Note that the energy estimate holds for a large class of kernels $K$ and nonlinearities $\Phi$.

\begin{thm}\label{thmlayerK1}
Suppose that $ u$ is a bounded monotone layer solution of (\ref{main}) when  $ F(1)=0$ and (\ref{Phidtb}) hold.   Assume also that the kernel $K$ satisfies either (\ref{Jumpki}) or  (\ref{Jumpkei})-(\ref{Jumpkeri}) with $D(r)< C r^{-\theta}$ when $\theta>\beta$. Then,     
\begin{equation}\label{EKRnminus1}
\mathcal E^\Phi_K(u,B_R) \le C  R^{n-1} \ \ \text{for} \ \  R>R_*, 
\end{equation}
where the positive constant $C$ is independent from $R$ but may depend on $R_*,\alpha,\beta$. 
\end{thm}

\begin{proof}
Set the shift function $u^t( x):=u( x',x_n+t)$ for $( x',x_n)\in\mathbb R^{n}$ and $t\in\mathbb R$.   The energy functional for the shift function $u^t$ is 
\begin{eqnarray*}
\mathcal E^\Phi_K(u^t,B_R) &=& \mathcal K^\Phi_K(u^t,B_R) - \int_{B_R}  F(u^t) dx  
\\&=&  \frac{1}{2} \int_{B_R} \int_{B_R}   \Phi[  u^t(x) -u^t(y) ] K(x-y) dy dx \\&&+  \int_{B_R} \int_{\mathbb{R}^n\setminus B_R}   \Phi[  u^t(x) -u^t(y) ]  K(x-y) dy dx - \int_{B_R}  F(u^t) dx, 
\end{eqnarray*}
where $R>R_*$.  We now differentiate the energy functional  in terms of parameter $t$ to get
\begin{eqnarray*}
\partial_t\mathcal E^\Phi_K(u^t,B_R) &=& \partial_t \mathcal K^\Phi_K(u^t,B_R) - \int_{B_R}  f' (u^t) \partial_t u^t dx  
\\&=& \frac{1}{2} \int_{B_R} \int_{B_R}   \Phi'[  u^t(x) -u^t(y) ]  [  \partial_t u^t(x) -\partial_t u^t(y) ]  K(x-y) dy dx \\&&+   \int_{B_R} \int_{\mathbb{R}^n\setminus B_R}  \Phi'[  u^t(x) -u^t(y) ]  [ \partial_t u^t(x) 
- \partial_t u^t(y) ] K(x-y) dy dx \\&&  - \int_{B_R}  f' (u^t) \partial_t u^t dx   . 
\end{eqnarray*}
Straightforward computations show that 
\begin{eqnarray} \nonumber
\partial_t\mathcal E^\Phi_K(u^t,B_R) &=&  
    \int_{\mathbb{R}^n\setminus B_R} \int_{B_R}  \Phi'[  u^t(x) -u^t(y) ]   \partial_t u^t(x)  K(x-y) dy dx  \\&& \label{tutT}+ \int_{B_R} \partial_tu^t(x) T_\Phi( u^t(x)) dx  - \int_{B_R}  f' (u^t) \partial_t u^t dx   .
\end{eqnarray}
It is straightforward to notice that $u^t$ is a solution of (\ref{main}).   Therefore, (\ref{tutT}) vanishes and consequently 
\begin{equation}
\partial_t \mathcal E^\Phi_K(u^t,B_R) =   \int_{\mathbb{R}^n\setminus B_R} \int_{B_R}  \Phi'[  u^t(x) -u^t(y) ]   \partial_t u^t(x)  K(x-y) dy dx    .
\end{equation}
Note that  $ \mathcal E^\Phi_K(u,B_R)= \mathcal E^\Phi_K(1,B_R)- \int_0^\infty \partial_t \mathcal E^\Phi_K(u^t,B_R) dt$.  From the fact that $ \mathcal E^\Phi_K(1,B_R)=0$, we obtain   
\begin{eqnarray*}\label{EKT}
\mathcal E^\Phi_K(u, B_R) &\le&   \int_{\mathbb{R}^n\setminus B_R} \int_{B_R}  \int_0^\infty | \Phi'[  u^t(x) -u^t(y) ] |  \partial_t u^t(x)  K(x-y) dt dy dx 
\\&\le&  \int_{\mathbb{R}^n\setminus B_R} \int_{B_R}  \int_0^\infty |  u^t(x) -u^t(y)  |^{\beta-1}  \partial_t u^t(x)  K(x-y) dt dy dx  ,
\end{eqnarray*}
where  (\ref{Phidtb}) is used.  Note that $|u^t(x)- u^t(y)| \le C |x-y|$. From the boundedness of $u$ and $|\nabla u|$,   we have  
\begin{equation*}
\mathcal E^\Phi_K(u, B_R) \le  C   \iint_{[(\mathbb{R}^n\setminus B_R )\times B_R ] }  |x-y|^{\beta-1}  K(x-y)  dy dx   . 
\end{equation*}
Therefore,   
\begin{equation}\label{Pik}
\mathcal E^\Phi_K(u, B_R) \le  C \sum_{k=1}^3  \iint_{\Pi^k_R }  |x-y|^{\beta-1} K(x-y)   dy dx
=: C \sum_{i=1}^3 I_i(R),  
\end{equation}
when 
\begin{equation}\label{Pi123}
\Pi^1_R:=B_{R-R_*}\times (\mathbb{R}^n\setminus B_R), \ \ \Pi^2_R:=B_{R}\times (\mathbb{R}^n\setminus B_{R+R_*}), \ \ \Pi^3_R:=(B_{R}\setminus B_{R-R_*})\times(B_{R+R_*}\setminus B_{R})  .  
\end{equation}
We first assume that (\ref{Jumpki}) holds.  Note that the above integrals $I_1(R)$ and $I_2(R)$,  on domains  $\Pi^1_R$ and $\Pi^2_R$,  vanish.  Hence, 
\begin{equation}\label{EphiBR}
\mathcal E^\Phi_K(u, B_R) \le  C  I_3(R)   \le   C  \int_{ B_{R}\setminus B_{R-R_*}   } \int_{  B_{R+R_*}\setminus B_{R}  }   |x-y|^{\beta-1-n-\alpha}   dy dx   . 
\end{equation}
On the other hand,  straightforward computations show that 
\begin{equation}\label{IntIER}
  \int_{ B_{R}\setminus B_{R-R_*}   } \int_{  B_{R+R_*}\setminus B_{R}  }   |x-y|^{\beta-1-n-\alpha}   dy dx 
\le  C \left\{ \begin{array}{lcl}
\hfill R_* R^{n-1}    \ \ \text{for}\ \ \ \beta- \alpha=1,\\   
\hfill \frac{(2R_*)^{-\alpha+\beta}}{(-1-\alpha+\beta)(-\alpha+\beta)} R^{n-1}    \ \ \text{for}\ \ \beta- \alpha \neq 1,
\end{array}\right.
\end{equation}
when $C$ is a positive constant it does not depend on $R,\alpha,\beta, R_*$.  Combining (\ref{IntIER}) and (\ref{EphiBR}) finishes the proof of (\ref{EKRnminus1}) for the truncated kernels satisfying (\ref{Jumpki}). We now assume that  (\ref{Jumpkei})-(\ref{Jumpkeri}) hold for $D(r)< C r^{-\theta}$  when $\theta>\beta$. 
\begin{equation*}
I_{1}(R) \le  C \left[ \int_{B_{R+R_*}\setminus B_R} dx \right]\left[ \sum_{k=1}^\infty \int_{k R_*<|z|<2k R_*} |z|^{\beta-1} K(z) dz\right]
\le C\left[ \sum_{k=1}^\infty k^{\beta-1-\theta} \right]  R^{n-1}   \le   C R^{n-1} , 
\end{equation*} 
where we have used  $D(r)< C r^{-\theta}$ for  $\theta>\beta$.    For  $I_{2}(R)$, we have 
\begin{eqnarray*}
I_{2}(R) &\le&   \int_{B_{R}}   \int_{|x-y|>R+\kappa_i- |x|}   |y-x|^{\beta-1} K(y-x)  dy dx 
\\&=& \int_{B_{R}}  \sum_{k=1}^\infty \int_{k(R+R_*- |x|)<|x-y|<2k(R+R_*- |x|)}   |y-x|^{\beta-1} K(y-x)  dy dx 
\\&\le&   \int_{B_{R}}  (R+R_* - |x|)^{\beta-1-\theta} dx  \left[ \sum_{k=1}^\infty k^{\beta-1-\theta} \right]\le C R^{n-1} \int_0^R  (R+R_* - r)^{\beta-1-\theta} dr
\\&=&C\left[ \frac{R_*^{\beta-\theta}}{\theta-\beta}  - \frac{(R+R_*)^{\beta-\theta} }{\theta - \beta} \right] R^{n-1} \le 
C\left[ \frac{R_*^{\beta-\theta}}{\theta-\beta} \right] R^{n-1}    ,
\end{eqnarray*} 
when $C$ is a positive constant that is independent from $R$.  Note that due to the structure of the domain $\Pi^3_R$, a similar estimate as  (\ref{IntIER}) holds for $I_3(R)$. This completes the proof.

\end{proof}

We end this section with an energy estimate for  layer solutions of (\ref{main}) when the kernel $K$ satisfies (\ref{Jumpi}) that is a generalization of the fractional Laplacian kernel. Note that in this case, unlike the previous theorem, the energy estimate depends on the exponent $\alpha$. For similar results in the case of fractional Laplacian operator where the Caffarelli-Silvestre extension problem is used we refer interested readers to \cite{cc,cc2}. Note that  our proofs do not rely on the local extension problem and we apply integral estimates directly,  as this is the case in \cite{cs,cp}. 

\begin{thm}\label{thmlayerK2}
Suppose that $ u$ is a bounded monotone layer solution of (\ref{main}) with  $ F(1)=0$ and (\ref{Phidtb}) holds.  Assume also that the kernel $K$ satisfies  (\ref{Jumpi}).   Then,  the following energy estimates hold for $R>\max\{R_*,1\}$. 
\begin{enumerate}
\item[(i)] If $0<\alpha<1$, then $\mathcal E_K(u,B_R)  \le  C R^{n-\alpha}$,
\item[(ii)] If  $\alpha=1$, then $\mathcal E_K(u,B_R) \le  C R^{n-1}\log R$,
\item[(iii)] If $\alpha>1$, then $\mathcal E_K(u,B_R)  \le  C R^{n-1}$,
\end{enumerate} 
where the positive constant $C$ is independent from $R$ but may depend on $R_*,\alpha,\beta$. 
\end{thm}

\begin{proof}
The proof is similar to the one of Theorem \ref{thmlayerK1}.   We only need to provide an upper bound for the right-hand side of (\ref{EKT}). From $|u^t(x)- u^t(y)| \le C \min\{R_*, |x-y|\}$ and the  boundedness of $u$,  we have  
\begin{eqnarray}\nonumber
\mathcal E^\Phi_K(u, B_R) &\le&   C    \iint_{(\mathbb{R}^n\setminus B_R )\times B_R  } \left[\min\{R_*, |x-y|\}\right]^{\beta-1} K (x-y)  dy dx   
\\&\le & 
\label{EPhiuBR}C  \iint_{\Pi_R}  \left[\min\{R_*, |x-y|\}\right]^{\beta-1} K (x-y)  dy dx  , 
\end{eqnarray}
where $\Pi_R$ is given by (\ref{Pi123}).  Note that an upper bound for the integral on $\Pi_R^3$ is given by (\ref{IntIER}).  Due to the symmetry in $\Pi^1_R$ and $\Pi^2_R$, we only compute an upper bound for the integral on $\Pi^1_R$ that is 
\begin{eqnarray*}
     R_*^{\beta-1}  \iint_{\Pi^1_R }|x-y|^{-n-\alpha}   dy dx
&=&   R_*^{\beta-1} \int_{B_{R-R_*}   } \int_{ \mathbb{R}^n\setminus B_R(x) }   |z|^{-n-\alpha}   dz dx
\\&\le&R_*^{\beta-1}  \int_{B_{R-R_*}   } \int_{R-|x|}^{\infty}   r^{-1-\alpha}   dr dx
\\&\le&  \frac{R_*^{\beta-1}}{\alpha} \int_{B_{R-R_*}   }  (R - |x|)^{-\alpha} dx   
\\&\le&   \frac{R_*^{\beta-1} }{\alpha} R^{n-1} \int_{0}^{R-R_*}     (R - r)^{-\alpha} dr   . 
\end{eqnarray*}
Straightforward computations show that the latter integral is bounded by the following term, 
\begin{equation}\label{IntJ2R11}
R_*^{\beta-1}\int_{B_{R-R_*}   } \int_{ \mathbb{R}^n\setminus B_R }   |x-y|^{-n-\alpha}   dy dx
\le  C   \left\{ \begin{array}{lcl}
\hfill  \frac{R_*^{\beta-1} }{\alpha} \log \left(\frac{R}{R_*} \right)  R^{n-1}  \ \ &\text{for}& \ \ \ \alpha=1,\\   
\hfill \frac{R^{\beta-1}_*}{\alpha(1-\alpha)} [R^{1-\alpha}- R^{1-\alpha}_*] R^{n-1}  \ \ &\text{for}&\ \ \alpha\neq 1 . 
\end{array}\right.
\end{equation}
Now combining (\ref{IntJ2R11}) and (\ref{EPhiuBR}) completes the proof. 
\end{proof}

\section{Sum of nonlocal operators}\label{SecSum}

This section is devoted to the sum of nonlocal and nonlinear operators as it is stated in \eqref{mainS}. The proofs are similar to the ones given in previous sections. Therefore, we omit the proofs.  The sum of fractional powers of Laplacian operators have been studied in the literature. We refer interested readers to \cite{cs} by Cabr\'{e} and Serra where symmetry results, among other interesting results, are provided via proving and applying the extension problem. In addition, Silvestre in \cite{si} studied  H\"{o}lder estimates and regularity properties for the sum operators. The following theorem states a Poincar\'{e} type inequality for the sum operators. 

\begin{thm}\label{thmpoinSum}
 Assume that  $n,m\ge 1$ and $ u$ is a stable solution of (\ref{mainS}).  Then,  
\begin{eqnarray*}
&& \sum_{i=1}^m \iint_{  \mathbb R^{2n}\cap  \{|\nabla_x u|\neq 0\}} \Phi_i''[u(x) - u(x+y)]     \mathcal A_y(\nabla_x u)  [\eta^2(x)+\eta^2(x+y)] K^i_{\Phi_i}(y) dx dy 
\\&\le&  \sum_{i=1}^m
\iint_{  \mathbb R^{2n}} \Phi_i''[u(x) - u(x+y)]  \mathcal B_y(\nabla_x u) [ \eta(x) - \eta(x+y) ] ^2 K^i_{\Phi_i}(y) d x dy  , 
  \end{eqnarray*} 
for any $\eta \in C_c^1(\mathbb R^{n})$ where 
\begin{eqnarray*}\label{mathcalA}
\mathcal A_y(\nabla_x u) &:= & |\nabla_x u(x)|  |\nabla_x u(x+y)| -\nabla_x u(x)  \cdot \nabla_x u(x+y) , 
\\ \label{mathcalB}
 \mathcal B_y(\nabla_x u)  & :=& |\nabla_x u(x)| | \nabla_x u(x+y)| .
 \end{eqnarray*}
 \end{thm}
 
 Applying the above Poincar\'{e}  inequality as well as other mathematical techniques we provide a one-dimensional symmetry result and a Liouville theorem as what follows. 
 
  \begin{thm}\label{thmoneDi}
Let $m\ge 1$ and $u$ be a bounded stable solution of  (\ref{mainS}) in two dimensions and  (\ref{Phiddtb}) for  each $\Phi_i$ and $\beta_i$.  Assume also that  the kernel $K$ satisfies  either (\ref{Jumpkis}) or (\ref{Jumpkeis}) and (\ref{Jumpkeris}) with $D_i(r)<C r^{-\theta_i}$ for $\theta_i>\beta_i+1$ for all $1\le i\le m$.    Then,  $u$ must be a one-dimensional function.   
  \end{thm}

 \begin{thm}\label{thmliouville}
Let $m\ge 1$ and  $ u$ be a bounded solution of (\ref{mainS}) when the kernel $K_i$ satisfies  either (\ref{Jumpkis}) or (\ref{Jumpkeis}) and (\ref{Jumpkeris}) with $D_i(r)<C r^{-\theta_i}$ for $\theta_i>\beta_i+1$ for all $1\le i\le m$.  If $f(u) \ge0$ or $uf(u)\le 0$, then $u$ must be constant provided   $n\le \min\{\beta_i, 1\le i\le m\}$.    
\end{thm}

Consider the following energy functional corresponding to (\ref{mainS}) 
\begin{equation*}
\mathcal E^\Phi_K(u,\Omega):=\sum_{i=1}^m \mathcal K^{\Phi_i}_{K_i}{(u,\Omega)} - \int_{\Omega}  F(u) dx, 
\end{equation*}
where each $\mathcal K^{\Phi_i}_{K_i}{(u,\Omega)} $ satisfies (\ref{kphisymm}) for even $\Phi_i$ and $K_i$.  Then, the following energy estimate holds for the sum operator when the kernel $K$ is of finite range or with decay at infinity.

\begin{thm}\label{thmlayerK1i}
Suppose that $ u$ is a bounded monotone layer solution of (\ref{main}) with  $ F(1)=0$ and (\ref{Phidtb}) hold.  Assume also that the kernel $K_i$ satisfies  either (\ref{Jumpkis}) or (\ref{Jumpkeis}) and (\ref{Jumpkeris}) with $D_i(r)<C r^{-\theta_i}$ for $\theta_i>\beta_i$ for all $1\le i\le m$.   Then,     
\begin{equation}\label{sumEKRn}
\mathcal E^\Phi_K(u,B_R) \le C  R^{n-1} \ \ \text{for} \ \  R>R_*:=\min\{R_i,1\le i\le m\},  
\end{equation}
where the positive constant $C$ is independent from $R$ but may depend on $R_i,\alpha_i,\beta_i$. 
\end{thm}

Lastly, the following theorem provides an energy estimate for layer solutions of (\ref{mainS}) where each kernel $K_i$ satisfies (\ref{alKal2is}). Note that, unlike the above, the following energy estimate depends on the minimum of all exponents $\alpha_i$. 

\begin{thm}\label{thmlayerK2i}
Suppose that $ u$ is a bounded monotone layer solution of (\ref{mainS}) with  $ F(1)=0$ and (\ref{Phidtb}) holds.  Assume also that the kernel $K_i$ satisfies  (\ref{alKal2is}).   Then,  the following energy estimates hold for $R>\max\{R_*,1\}$. 
\begin{enumerate}
\item[(i)] If $0<\alpha_*<1$, then $\mathcal E_K(u,B_R)  \le  C R^{n-\alpha_*}$,
\item[(ii)] If  $\alpha_*=1$, then $\mathcal E_K(u,B_R) \le  C R^{n-1}\log R$,
\item[(iii)] If $\alpha_*>1$, then $\mathcal E_K(u,B_R)  \le  C R^{n-1}$,
\end{enumerate} 
where $\alpha_*:=\min\{\alpha_i, 1\le i\le m\}$ and   the positive constant $C$ is independent from $R$ but may depend on $R_i,\alpha_i,\beta_i$. 
\end{thm}

\noindent {\it Acknowledgment}.   The  authors would like to thank the anonymous referee(s) for comments and for pointing out the reference \cite{bucur}.

\end{document}